\documentclass[12pt,a4paper]{amsart}
\usepackage{amsthm} 
\usepackage{amsfonts}
\usepackage{amsmath}
\usepackage{amssymb}
\usepackage{dsfont}
\usepackage{txfonts}
\usepackage{cancel}
\usepackage{soul}
\usepackage{mathrsfs} 
\usepackage{fullpage}
\usepackage{enumerate}
\usepackage{textcomp}
\usepackage{hyperref} 
\usepackage{color}

\author
{Charu Goel 
and Bruce Reznick}
\address{Department of Basic Sciences, Indian Institute of Information Technology Nagpur, 440006, 
India}
\email{charugoel@iiitn.ac.in}

\address{Department of Mathematics, University of
Illinois at Urbana-Champaign, Urbana, IL 61801}
\email{reznick@math.uiuc.edu}
 
\title[Sums of squares 
of $k$-term forms]
{Sums of squares 
of $k$-term forms\ }
\date{}

\keywords{positive polynomials, sums of squares, sums of $k$-nomial squares}
\subjclass[2010]{11E76, 11E25, 05A19, 05A10, 11B65}

\numberwithin{equation}{section}

\begin{document}

\theoremstyle{definition}
  \newtheorem{thm}{Theorem}[section]
 \newtheorem{df}[thm]{Definition}
 \newtheorem{nt}[thm]{Notation}
 \newtheorem{prop}[thm]{Proposition}
 \newtheorem{rk}[thm]{Remark}
 \newtheorem{lem}[thm]{Lemma}
 \newtheorem{obs}[thm]{Observation}
 \newtheorem{cor}[thm]{Corollary}
  \newtheorem{conj}[thm]{Conjecture}
 \newtheorem{nthng}[thm]{}
\newcommand{\al}{\alpha}
\newcommand{\be}{\beta}
\newcommand{\ga}{\gamma}
\newcommand{\la}{\lambda}


\begin{abstract}
In this paper we study the cones corresponding to sums of squares of $n$-ary $d$-ic forms with at most $k$ terms. We show that these are strictly nested as $k$ increases, leading to the usual sum
of squares cone. We also discuss the duals of these cones. For $n \ge 3$, we construct
indefinite irreducible $n$-ary $d$-ic forms with exactly $k$ terms for $2 \le k \le \binom{n+d-1}{n-1}$.
\end{abstract} 

\maketitle

\section{Introduction}
A real form (homogeneous polynomial) $f$ is
called \textit{positive semidefinite} (psd) if it takes only
non-negative values and it is called a \textit{sum of squares} (sos)
if there exist real forms $h_j$ so that $f = h_1^2 + 
\cdots + h_k^2$. Every sos form must be psd, and every psd form must have even degree $2d$.

More formally, let $F_{n,d}$ denote the vector space of real forms of degree $d$ in $n$ variables; $p \in F_{n,d}$ is called a  $n$-ary $d$-ic form. Following Choi and Lam \cite{CL_1}, we write
\[
P_{n,2d} = \left \{ f \in F_{n,2d}  \ \big \vert \ \text{$f$ is psd}\  \right \}, \qquad 
\Sigma_{n,2d} = \left\{ \sum_{j=1}^m h_j^2, \ \big \vert  \ m \in \mathbb N, h_j \in F_{n,d} \right \}.
\]
The sets $P_{n,2d}$ and $\Sigma_{n,2d}$ are both closed convex cones.

Several authors have also studied sums of squares
of forms with a restricted number of terms. More formally, let $F_{n,d}^k \subset 
F_{n,d}^{}$ denote the subset of $n$-ary $d$-ics with at most $k$ terms, and let  
\[
\Sigma_{n,2d}^k = \left\{ \sum_{j=1}^m h_j^2, \ \big \vert  \ m \in \mathbb N, h_j \in F_{n,d}^k \right \}.
\]

There are only $N(n,d) = \binom{n+d-1}{n-1}$ possible terms in
an $n$-ary $d$-ic form, hence $\Sigma_{n,2d}^k = \Sigma_{n,2d}^{}$ for $k \ge N(n,d)$. Since 
$F_{n,d}^1$ consists of the $n$-ary $d$-ic monomials, 
\[
f \in \Sigma_{n,2d}^1 \iff f(x_1,\dots,x_n) = \sum_{j=1}^m \alpha_j x_1^{2c_{j1}} \cdots x_n^{2c_{jn}} 
=\sum_{j=1}^m \alpha_j x^{2c_j} \in F_{n,2d}, \ \text{where} \ \alpha_j \ge 0.
\]
That is, $f \in \Sigma_{n,2d}^1$ if and only if it is an even $n$-ary $2d$-ic (each term is even in 
each variable) with non-negative coefficients. (For simplicity, we use the multinomial notation $x^{\alpha} = x_1^{\alpha_1} \cdots x_n^{\alpha_n}$.)

A form in $F_{n,2d}$ is called \textit{even symmetric} if it is an even form as well as a symmetric form. 
Necessary and sufficient conditions were given in \cite{CLR-3} for an even
symmetric sextic to be psd and to be sos.  In addition, the term \textit{sum of binomial squares} 
(\textit{sobs}) was defined there to mean a sum of terms $(a x^{\alpha} + b x^{\beta})^2$, and a 
necessary and sufficient condition was given for an even symmetric sextic to be sobs. In our
notation, the set of all sobs $n$-ary $2d$-ics is $\Sigma_{n,2d}^2$. 

Many examples of psd forms in the literature have been constructed by substitution into the
arithmetic-geometric inequality. In \cite{Hu}, Hurwitz proved that if $0 \le a_j \in \mathbb Z$ and
$\sum_j a_j = 2d$, then
\[
a_1x_1^{2d}+ \cdots + a_nx_n^{2d} - 2dx_1^{a_1}\cdots x_n^{a_n} \in \Sigma_{n,2d}^2.
\]
(See also \cite{Rez-3}.) 
More generally, one can put even monomials into the arithmetic-geometric
inequality to obtain what is called an \textit{agiform} (see \cite{Rez-4}). A typical agiform has the
shape
\[
\lambda_1 x^{\alpha_1} + \cdots + \lambda_m x^{\alpha_m} - x^{\sum \lambda_j\alpha_j},
\]
where $0 \le \lambda_j, \sum \lambda_j = 1$,  $\alpha_j, \sum \lambda_j\alpha_j \in \mathbb Z^d$
and each $x^{\alpha_j}$ is an even $n$-ary
$2d$-ic monomial. (For example, the Motzkin form is an agiform with 
$\lambda = (\frac 13,\frac 13, \frac 13)$ and monomials $x_1^4x_2^2, x_1^2x_2^4, x_3^6$.)
Agiforms must be psd. A certain lattice point property on the polytope $cvx(\{\alpha_j\})$
defined by the exponents (see \cite{P-R}) determines  whether an agiform is sos, and if so, then the agiform
is in $\Sigma_{n,2d}^2$. Further,
 for any agiform, there exists $T$ so that, for $r \ge T$,
\[
\lambda_1 x^{r\alpha_1} + \cdots + \lambda_m x^{r\alpha_m} - x^{\sum r \lambda_j\alpha_j}
\in \Sigma_{n,2rd}^2,
\]
so that every agiform can be expressed as a sum of binomial squares in the variables 
$\{x_j^{1/r}\}$.

Agiforms have been generalized by several authors and can be regarded as a special
case of \textit{circuit polynomials}  (see \cite{I-W}). Conditions under which such a form is in $\Sigma_{n,2d}^2$
have been studied by Fidalgo, Kova\v{c}ec, Ghasemi and Marshall in  \cite{Fi-Ko, GhMar-2}.
Recently, Ahmadi and Majumdar \cite{Ahmadi-Majumdar}  discussed optimization and its 
application to $\Sigma_{n,2d}^2$, which they call  SDSOS$_{n,2d}$.  Optimization over
$\Sigma_{n,2d}$ is computationally much faster than optimization over $P_{n,2d}$ and often
gives the right answer. Similarly, optimization over SDSOS$_{n,2d}$ is even faster. 
In \cite[\S3]{Ahmadi-Majumdar},
they discuss multipliers; cases where $p \notin \Sigma_{3,6}^2$ (or even  $\Sigma_{3,6}$), but
$(\sum x_j^2)^rp \in  \Sigma_{3,6+2r}$ for $r=1,2$. They reinterpret P\'olya's Theorem: if $p$ is a 
strictly definite even form in $P_{n,2d}$, then for sufficiently large $r$, 
$(\sum x_j^2)^rp \in  \Sigma_{n,d+2r}^1$. (For more on Polya's Theorem, see \cite{P-R-3}.)

Gouvea, Kova\v{c}ec and Saee \cite{GouvieaKovacecSaee'19} study the question: under what conditions 
does $p \in \Sigma_{n,2}$ imply that $(\sum x_j^2)^rp \in  \Sigma_{n,2+2r}^k$?
They show (Thm. 1.1) that if $p$ is symmetric or $k=2$ or $(n,k) = (4,3)$, then this is true if
and only if $p \in  \Sigma_{n,2+2r}^k$. However, there is some evidence that this result fails for
$(n,k) = (5,3)$. They also discuss the dual cones $(\Sigma_{n,2d}^{k})^*$ (see below), using somewhat
different language than that found in this paper.

In this paper, we consider $\Sigma_{n,2d}^{k}$ as an object in its own right. 
Clearly, $\Sigma_{n,2d}^{k}$ is a convex cone. Our main results are these:

\begin{thm}\label{T:thmA}
For all $n,d,k$, $\Sigma_{n,2d}^{k}$ is a closed convex cone.
\end{thm}

\begin{thm}\label{T:thmB}
We have $\Sigma_{n,2d}^{k-1} \subsetneqq \Sigma_{n,2d}^{k}$ for $2 \le k \le N(n,d)$.
\end{thm}

\

In some cases, it is not hard to illustrate
 this strict inclusion. For $k=1$,  if $p(x_1,\dots,x_n) = x_1^{d-1}(x_1+x_2)$, then
 $p(x)^2$ is evidently in 
$\Sigma_{n,2d}^2 \setminus \Sigma_{n,2d}^1$. This example provides a template for the
results we shall present here. More generally, we shall rely on a result which can be 
found in \cite[Theorem 4.5.1] {BCR} or \cite[Theorem 1.5]{Goel}.
\begin{prop}
If $h \in F_{n,d}^k$ is an indefinite irreducible form, then 
\begin{equation}\label{E:irredsquare}
h^2 = \sum_{j=1}^m h_j^2 \iff h_j = \lambda_j h,\quad  \sum_{j=1}^m \lambda_j^2 =1.
\end{equation}
The same is true if $h$ is a product of linear forms.
\end{prop}

\begin{cor}\label{C:square}
If $h \in F_{n,d}^k \setminus F_{n,d}^{k-1} $ is an indefinite irreducible form or a product of linear forms, then $h^2 \in \Sigma_{n,2d}^{k}\setminus \Sigma_{n,2d}^{k-1}$.
\end{cor}

Here is a quick proof of Theorem \ref{T:thmB} in the special case that $n=2$.
\begin{proof}
For $1 \le k \le d+1=N(2,d)$, consider the product 
$g_k(x) = x_1^{d-k+1}(x_1+x_2)^{k-1} \in F_{n,d}^{k}$; then 
$g_k^2 \in \Sigma_{n,2d}^{k} \setminus \Sigma_{n,2d}^{k-1}.$
\end{proof}

What follows is the organization of the paper.
In section two, we review some notation for forms and use this to 
give a proof of Theorem  \ref{T:thmA}. Under the ``standard" Fischer inner product, we describe the
dual cone, $(\Sigma_{n,2d}^{k})^*$. It follows from Theorem \ref{T:thmB} that 
\begin{equation}\label{E:dual}
(\Sigma_{n,2d}^{k-1})^* \supsetneqq (\Sigma_{n,2d}^{k})^*, 
\end{equation}
We give some explicit examples of \eqref{E:dual} in the easiest case: $(n,2d) = (2,4)$.

In section three, we prove Theorem  \ref{T:thmB} for $n \ge 3$. In view of Corollary \ref{C:square}, it suffices to
find indefinite irreducible $n$-ary $d$-ics with exactly $k$ terms for $n \ge 3$ and $2 \leq k \leq N(n,d)$. For $k \ge 3$, the ones we use
are  small perturbations of $x_1^d - x_2^d + x_3^d$.
Such a perturbation is evidently indefinite. The irreducibility follows from  the well-known (but hard to cite)
theorem from algebraic geometry that the irreducible forms comprise an open set. We include a self-contained proof
of this fact as Theorem  \ref{irred}. The paper concludes with some open questions.

\section{The cone \texorpdfstring{$\Sigma_{n,2d}^k,$} \ \  its dual, and the proof of Theorem \ref{T:thmA}}

Much of the background material in this section is adapted from Chapter 3 of \cite{Rez-1}; see
also \cite{Rez-2}. We begin with some notation for forms. It is sensible to consider the more general set $\mathcal F_{n,d}$ of 
forms in $\mathbb C[x_1,\dots,x_n]$ of degree $d$, so that $F_{n,d}$ is just the subset of $\mathcal F_{n,d}$ with real coefficients. 

 The index set consists of $n$-tuples of non-negative integers:  
\begin{equation}
 \mathcal I(n,d) = \biggl\lbrace i=(i_1,\dots,i_n): \sum\limits_{k=1}^n
 i_k = d\biggr\rbrace.
\end{equation}
We have  $\binom {n+d-1}{n-1} = N(n,d) = |\mathcal I(n,d)|$;  for $i
\in \mathcal I(n,d)$, let 
$c(i) = \frac{d!}{i_1!\cdots i_n!}$ be the associated multinomial coefficient.
Every $p \in \mathcal F_{n,d}$ can be written as
 \begin{equation}
p(x_1,\dots,x_n)=\sum_{i\in\mathcal I(n,d)} c(i)a(p;i)x^i,\ c(i) \in \mathbb C.
\end{equation}

We also define the \textit{support} of
$p$, $S(p)$,  to be the set of exponents of monomials whose exponents appear in $p$:
\[
S(p) = \{i \in {\mathcal I(n,d)} \ | \ a(p;i) \neq 0\}.
\] 
It follows immediately that $p \in F_{n,d}^k  \iff |S(p)| \le k$.

We identify a form $p \in \mathcal F_{n,d}$ with its tuple of coefficients $(a(p;i)) \in \mathbb C^{N(n,d)}$, and similarly,
a form $p \in F_{n,d}$ with its tuple of coefficients $(a(p;i)) \in \mathbb R^{N(n,d)}$. 
In this sense, $P_{n,2d}$ and $\Sigma_{n,2d}$ are convex cones in 
$\mathbb R^{N(n,2d)}$. The usual topology for convex cones in
$\mathbb C^{N(n,d)}$ pulls back to the usual topology for forms in $\mathcal F_{n,d}$:
\[
p_n \to p \iff a(p_n;i) \to a(p;i) \quad \text{for all $i \in \mathcal I(n,d)$}.
\]

The classical Fischer inner product (see \cite{Rez-1}) has wide applications:
For $p$ and $q$ in $\mathcal F_{n,d}$, let 
\begin{equation}\label{E:ip}
[p,q] = \sum_{i \in \mathcal I(n,d)} c(i)a(p;i)\overline{a(q;i)}. 
\end{equation}
This can be used to define a norm on $\mathcal F_{n,d}$ which we use in the next section:
\begin{equation}\label{E:norm}
||p||^2 = [p,p] = \sum_{i \in \mathcal I(n,d)} c(i)|a(p;i)|^2. 
\end{equation}
Since $c(i) \ge 1$ for $i \in \mathcal I(n,d)$, we obtain the following trivial bound:
\begin{equation}\label{E:bound}
||p|| \le M \implies |a(p;i)| \le \sqrt{M}.
\end{equation}
This applies equally well to $F_{n,d}$, where the $a(p;i)$'s and $a(q;i)$'s are real. 

There is an elementary analytic lemma which will be useful in the proof of Theorems 
\ref{T:thmA} and \ref{T:thmB} 
and for the sake of completeness, we include a proof.

\begin{lem}\label{L:BW}
Let $(F_m)$,  $F_m = (f_{m,1}, \dots, f_{m,v})$, $f_{m,\ell} \in \mathcal F_{n,d_{\ell}}$, be a sequence of 
$v$-tuples of forms
such that, for all $m$, $1 \le \ell \le v$, and  $i \in \mathcal I(n,d_{\ell})$, 
we have $|a(f_{m,\ell};i)| \le T$ for some $T > 0$. 
Then there exists a convergent subsequence 
$(F_{m(t)})$; that is, for all $\ell$,
there exist $f_{\ell} \in \mathcal F_{n,d_\ell}$ so that  $f_{m(t),\ell} \to f_{\ell}$. Furthermore, if all forms are
real and 
each $f_{m,\ell} \in F_{n,d_\ell}^k$, then $f_\ell \in F_{n,d_\ell}^k$ as well. 
\end{lem}
\begin{proof}
The first assertion is an immediate consequence of Bolzano-Weierstrass. Consider the 
$M$-tuple $x_m \in \mathbb C^M$ (here, $M = \sum N(d_\ell)$) 
consisting of the concatenated $a(f_{m,\ell},i)$'s. Let $B_T = \{z: ||z|| \le T\}$ denote the (compact) closed ball
of radius $T$ in $\mathbb C^M$.
By hypothesis, $x_m \in B_T^M$, and so there is a convergent
subsequence $(x_{m(t)})$ of coefficients, implying the desired convergent subsequences of forms. 

Now suppose forms are real and 
suppose $a(f_{\ell};i) \neq 0$. Then for $t \ge t_i$, $a(f_{m(t),\ell};i) \neq 0$ as well.
Considering this over all
$i \in S(f_{\ell})$ shows that for sufficiently large $t$, $|S(f_{\ell})| \le 
 |S(f_{m(t),\ell})| \le k$. 
\end{proof}

\begin{proof}[Proof of Theorem \ref{T:thmA}]
Suppose $p_m \in \Sigma_{n,2d}^{k}$ and $p_m \to p$ with
\[
p_m(x) = \sum_{\ell = 1}^{R(m)} g_{m,\ell}^2(x), \quad  g_{m,\ell} \in F_{n,d}^k.
\]

Following the argument of \cite[p.27,37]{Rez-1}, since for every $m$, $\{g_{m,\ell}^2\} \subset
F_{n,2d}$, which is an $N(n,2d)$-dimensional vector space, 
Carath\'eodory's Theorem implies that
there is a subset of at most $N(n,2d)$ $g_{m,\ell}^2$'s and
 $\beta_{m,u}\ge 0$ so that
\[
\sum_{\ell = 1}^{R(m)} g_{m,\ell}^2(x) = 
\sum_{u= 1}^{N(n,2d)} \beta_{m,u} \cdot g_{m,\ell_u}^2(x).
\]
Thus, we may assume 
 that each $p_m$ is already a sum of at most $N(n,2d)$ 
squares: 
\[
p_m(x) = \sum_{\ell = 1}^{N(n,2d)} f_{m,\ell}^2(x), \quad  f_{m,\ell} \in F_{n,d}^k.
\]

In addition, as $p_m(x) \to p(x)$, there is a uniform bound to $|p_m(x)|$ for $x \in T$, for any
finite set $T$, and since $f_{m,\ell}^2(x) \le p_m(x)$, this implies that there is a uniform bound
on  $|f_{m,\ell}(x)|$ for $x \in T$.
 As argued in \cite[p.30,37]{Rez-1}, this condition  implies that there is a uniform
bound on the $|(a(f_{m,\ell};i)|'s$. 

Consider now the $N(n,2d)$-tuple $F_m= (f_{m,1}, \dots, f_{m,N(n,2d)})$. By Lemma \ref{L:BW}
there is a convergent subsequence $F_{m(t)}$ so that for $1 \le \ell \le N(n,2d)$,
\[
f_{m(t),\ell} \to f_\ell \in F_{n,d}^k.
\]
Thus,
\[
p = \lim_{t \to \infty}  p_{m(t)} 
  = \lim_{t \to \infty} \sum_{\ell = 1}^{N(n,2d)} f_{m,\ell}^2(x) 
=  \sum_{\ell = 1}^{N(n,2d)}\lim_{t \to \infty} f_{m(t),\ell}^2
= \sum_{\ell = 1}^{N(n,2d)} f_{\ell}^2 \in \Sigma_{n,2d}^k.
\]
\end{proof}

Since we now know that $F_{n,d}^k$ is a real closed convex cone, it makes sense to look at its dual. 
(In the rest of this section, we shall only discuss real forms.) 
Write 
\begin{equation}
h(x) = \sum_{i \in \mathcal I(n,d)} t(i)x^i,
\end{equation}
then a simple computation (see \cite[pp.6-7]{Rez-1})
shows that 
\begin{equation}
 H_p(t):= [p,h^2]  = \sum_{i \in \mathcal I(n,d)}\sum_{j \in \mathcal I(n,d)}
a(p;i+j)t(\ell_i)t(\ell_j)
\end{equation}
is a quadratic form in the coefficients of $h$.

Note that the coefficients of the quadratic form $H_p$ are restricted by the condition
that if $i,j,i',j' \in \mathcal I(n,d)$ and $k = i + j = i' + j'$, then the coefficients of 
$t(\ell_i)t(\ell_j)$ and $t(\ell_i')t(\ell_j')$ are both equal to $a(p;k)$. This makes
$H_p$ a {\it generalized
Hankel quadratic form}. 

Thus, $p \in \Sigma_{n,2d}^*$ if and only if the quadratic
form $H_p$ is psd 
(see \cite[pp.40-41]{Rez-1},
\cite[pp.350-355]{Rez-2}). An equivalent condition is that the matrix
associated to the quadratic form $H_p$ is psd. 

Because of the restriction on the summands in $\Sigma_{n,2d}^k$, we have by definition
\begin{equation}
 \left({{\Sigma}_{n,2d}^k}\right)^* = \{ p \in F_{n,2d}\ : \ [p, h^2] \ge 0\ \text{for
   all } h \in F_{n,d}^k\}.
\end{equation}
To be specific, suppose $\{i_1, \dots, i_k\} \subset \mathcal I(n,d)$ and
\[
h(x) = \sum_{m= 1}^k t(i_m) x^{i_m} \in F_{n,d}^k,
\]
then
\[
[p, h^2] = \sum_{m,m'} a(p;i_m + i_m') t(\ell_{i_m})t(\ell_{i_m'}).
\]
Then $p \in \Sigma_{n,2d}^*$ if and only if this expression is psd in the variables
$t(i_1), \dots t(i_k)$ for every choice of a subset of $k$ elements from $ \mathcal I(n,d)$.
In other words, we have proved the following theorem; cf. \cite[Prop. 3.3]{GouvieaKovacecSaee'19}:
\begin{thm}
For $p \in F_{n,2d}$, construct the matrix associated to the quadratic form $H_p$:
\[
M_p = [a(p;i+j): i,j \in \mathcal I(n,d)].
\]
Then $p \in (\Sigma_{n,2d}^k)^*$ if and only if every $k \times k$ principal  submatrix of $M_p$ is
psd.
\end{thm}

It is easy to see that there exist quadratic forms
 which satisfy this criterion for $k$ and not for $k+1$.
For example, let 
\[
q(x_1,\dots,x_m) = \lambda \sum_{i=1}^m x_i^2 -
\left(\sum_{i=1}^m x_i\right)^2.
\]
No matter which $k$ variables $u_i = x_{n_i}$ are chosen for the substitution, one
obtains
\[
 \lambda \sum_{i=1}^k u_i^2 -\left(\sum_{i=1}^k u_i\right)^2.
\]
By Cauchy-Schwartz, this is psd if and only if $\la \ge k$. Thus, if
we choose $\la = k + \frac 12$, we see that $q$ has the indicated
property. Unfortunately, these quadratic forms will not be $H_p$ for
any $p$, because they are not Hankel. 

The simplest nontrivial case is $(n,d) = (2,4)$. For $p \in F_{2,4}$,  
write $b_i = a(p;(4-i,i))$. That is,
\begin{equation}\label{E:241}
p(x,y) = \sum_{i=0}^4 \binom 4i b_ix^{4-i}y^i, \qquad M_p = 
\begin{pmatrix} b_0 & b_1 & b_2 \\ b_1 & b_2 & b_3 \\
b_2 & b_3 & b_4   \end{pmatrix}.
\end{equation}

It is easy to write down the criteria for $p \in (\Sigma_{2,4}^k)^*$ for $k=1,2$:
\begin{equation}\label{E:242}
\begin{gathered}
p \in (\Sigma_{2,4}^1)^* \iff b_0,\ b_2,\ b_4 \ge 0; \\
p \in (\Sigma_{2,4}^2)^*\iff b_0,\ b_2,\ b_4, \ \ b_0b_2 -b_1^2,\  b_0b_4-b_2^2,\ 
 b_2b_4 - b_3^2 \ge 0.
 \end{gathered}
\end{equation}
For  $(\Sigma_{2,4}^3)^* = \Sigma_{2,4}^*$, we add the condition 
 that  $\det(M_p) \ge 0$.

We have already seen that $(x(x+y))^2 \in \Sigma_{2,4}^2 \setminus  \Sigma_{2,4}^1$, and
$((x+y)^2)^2 \in \Sigma_{2,4}^3 \setminus  \Sigma_{2,4}^2$, 
 hence
there must exist elements in $(\Sigma_{2,4}^1)^* \setminus (\Sigma_{2,4}^2)^*$
and 
$(\Sigma_{2,4}^2)^* \setminus (\Sigma_{2,4}^3)^*$.

Here are some illustrations of this separation. 
Let $p_1(x,y) = x^4 - 4x^3y$, then
\[
\begin{gathered}
M_{p_1} = \begin{pmatrix} 1& -1 &0 \\ -1 & 0 & 0 \\
0 & 0 & 0   \end{pmatrix} \implies p_1 \in (\Sigma_{2,4}^1)^* \setminus  (\Sigma_{2,4}^2)^*; \\
[p_1,(t_0x^2 + t_1 xy + t_2y^2)^2] = t_0^2 - 2t_0t_1 \implies [p_1,(x^2 + xy)^2] = -1 < 0.
\end{gathered}
\]
The second example is a bit trickier. Let $p_2(x) = 4x^4 - 8x^3y + 6x^2y^2-8xy^3 + 4y^4$, so that
\[
\begin{gathered}
M_{p_2} = \begin{pmatrix} 4& -2 &1 \\ -2 & 1 & -2 \\
1& -2 & 4   \end{pmatrix} \implies (\Sigma_{2,4}^2)^* \setminus  (\Sigma_{2,4}^3)^*,\quad
\text{since $\det(M_{p_2}) = - 9$}; \\
[p_2,(t_0x^2 + t_1 xy + t_2y^2)^2] = 4t_0^2 - 4t_0t_1 + t_1^2 +2t_0t_2 -4t_1t_2 + 4t_2^2 \\
\implies [p_2,(x^2 + 2 xy +y^2)^2] = -2 < 0.
\end{gathered}
\]

A closed convex cone $C\subset \mathbb R^N$ is called {\it pointed} if $x, - x \in C \implies x = 0$; otherwise,
$C$ is {\it flat}. The significance of this property is that if $C$ is a pointed cone, then every element in $C$ can
be written as a sum of extremal elements. 

It is easy to describe the elements of the cone $(\Sigma_{n,2d}^1)^*$: they are the 
$n$-ary $2d$-ic forms $p$ with the property that if $x^i$ is a square, then $a(p;i) \ge 0$. 
This cone is not pointed but is flat: if $x^j$ is {\it not} a square, then $\pm x^j \in (\Sigma_{n,2d}^1)^*$. 
However, this is not true for $k \ge 2$.

\begin{lem}
If $k \ge 2$, then $(\Sigma_{n,2d}^k)^*$ is pointed.
\end{lem}
\begin{proof}
Suppose $q, -q \in (\Sigma_{n,2d}^k)^*$. Then for $p \in \Sigma_{n,2d}^k$, we have $[p,q]
\ge 0$ and $[p,-q] = - [p,q] \ge 0$, hence $[p,q] = 0$. Any $i \in \mathcal I(n,2d)$ may be
written as $i = j + j'$ where $j, j' \in  \mathcal I(n,d)$ (just write $x^i = x^jx^{j'}$) and since
$(x^j + x^{j'})^2 - (x^j - x^{j'})^2 = 4x^{j+j'} = 4x^i$, 
\[
0 = [(x^j + x^{j'})^2,q] =  [(x^j - x^{j'})^2,q] \implies 0 = [4x^i,q] = 4a(q;i).
\]
Since $a(q;i) = 0$ for all $i \in \mathcal I(n,2d)$, $q=0$ and 
thus, $(\Sigma_{n,2d}^k)^*$ is pointed. 
\end{proof}

It is a natural question to characterize the extreme elements of $(\Sigma_{n,2d}^k)^*$
when $k \ge 2$. For $n=2$, $\Sigma_{2,2d}^{N(2,d)} = \Sigma_{2,2d} = P_{2,2d}$, so the
dual cone is generated by the $2d$-th powers of linear forms (see  \cite{Rez-1}). 

We now discuss the extremal elements of $(\Sigma_{2,4}^2)^*$. Using the notation
\eqref{E:241}, let
\[
q(b_0,b_1,b_2,b_3,b_4)(x,y) = \sum_{i=0}^4 \binom 4i b_ix^{4-i}y^i.
\]
\begin{thm}
Up to positive multiple, the extremal elements of $(\Sigma_{2,4}^2)^*$ are:
\[
\{x^4, y^4\} \cup  \{q(r^2,\epsilon_1 rt, t^2,\epsilon_2 st,s^2)(x,y)\ : \ r > 0, s > 0, \sqrt{rs} \ge t > 0,
 \epsilon_j \in \{\pm1\} \}.
\]
\end{thm}

\begin{proof}
By \eqref{E:242}, we have $b_0,b_2,b_4 \ge 0$. If $b_0 = 0$, then immediately, $b_1= b_2 = 0$
and so $b_3 = 0$. Since $b_4\ge 0$, we obtain $b_4y^4$. If $y^4 = q_1(x,y) + q_2(x,y)$, where
$q_i \in (\Sigma_{2,4}^2)^*$, then the coefficient of $x^4$ in $q_i$ must be zero, so by the first
argument, $q_i$ must be a multiple of $y^4$; thus $y^4$ is extremal. Similarly, if $b_4 = 0$,
then $b_3= b_2 = 0$, so $b_1 = 0$ and we obtain $b_0 x^4$. If $b_2 = 0$, then $b_1= b_3 = 0$,
and we obtain $b_0x^4 + b_4y^4$, which is only extremal if $b_0 = 0$ or $b_4 = 0$.

Otherwise, we may write $b_0 = r^2, b_2 = t^2, b_4 = s^2$, with $r,s,t > 0$ and $r^2s^2 \ge t^4$,
so $\sqrt{rs} \ge t$. If we fix $(r,t,s)$, the remaining conditions are that $r^2t^2 \ge b_1^2$ and
$s^2t^2 \ge b_3^2$. Thus $(b_1,b_3)$ lies in the rectangle with vertices $(\pm rt, \pm st)$ and
so $q(r^2,b_1,t^2,b_3,s^2)$ is a convex combination of the four forms
$q(r^2,\epsilon_1 rt, t^2,\epsilon_2 st,s^2)$. We now prove that each such form is extremal.

First observe that
\[
\begin{gathered}
\ [q(r^2,\epsilon_1 rt, t^2,\epsilon_2 st,s^2)(x,y),(tx^2 - \epsilon_1 rxy)^2] =
r^2t^2 - 2\epsilon_1^2(rt)^2 + t^2r^2 = 0, \\
[q(r^2,\epsilon_1 rt, t^2,\epsilon_2 st,s^2)(x,y),(ty^2 - \epsilon_2 sxy)^2] =
r^2s^2 - 2\epsilon_2^2(st)^2 + t^2s^2 = 0 
\end{gathered}
\]
Thus, if $q(r^2,\epsilon_1 rt, t^2,\epsilon_2 st,s^2) = q_1 + q_2$ with $q_i \in (\Sigma_{2,4}^2)^*$,
then
\[
[q_i,(tx^2 - \epsilon_1 rxy)^2] = [q_i, (ty^2 - \epsilon_2 sxy)^2] = 0.
\]
Write $q_1 =  \sum_{i=0}^4 \binom 4i b_ix^{4-i}y^i$, suppose $b_2 = \alpha t^2$, and 
consider the quadratic 
\[
\phi(\lambda): = [q_1,(\la x^2 - \epsilon_1 rxy)^2]= b_0 \la^2 - 2b_1 \epsilon_1 r\la + b_2 r^2 =
 b_0 \la^2 - 2b_1 \epsilon_1 r\la + \alpha t^2r^2 \ge 0.
\]
Since $\phi(t) = \phi'(t) = 0$, it follows that $b_0 = \alpha r^2$ and $b_1 = \alpha \epsilon_1rt$.
A similar consideration of $[q_1,(\la y^2 - \epsilon_1 rxy)^2]$ shows that $b_3 = 
\alpha \epsilon_1st$ and $b_4 \alpha s^2$; that is, 
$q_i = \alpha q(r^2,\epsilon_1 rt, t^2,\epsilon_2 st,s^2)$, proving extremality. 
\end{proof}

\section{Proof of Theorem \ref{T:thmB} }
We still must prove Theorem \ref{T:thmB} for $n \geq 3$, and in order to do so, we construct a large class of indefinite irreducible forms in $\mathcal F_{n,d}$.

It seems to be universally known among algebraic geometers that the set of irreducible forms in
$\mathcal F_{n,d}$ is open, i.e. if $p \in \mathbb C[x_1,\dots,x_n]$ is an irreducible form of degree d, then every form found by making a sufficiently small perturbation of the coefficients of $p$ is also irreducible. At the same time, finding an explicit statement of this humble theorem
and its proof seems hard to find.

Recall that for $p, q \in \mathcal F_{n,d}$, the Fischer inner product is defined by \eqref{E:ip}.
There is another interpretation for this inner product. 
 A form $p \in \mathcal F_{n,d}$ naturally defines
 a differential operator $p(D)$ of order $d$ by replacing each appearance
of the variable $x_j$ with $\frac{\partial}{\partial x_j}$. For $i\in\mathcal I(n,d)$, define
\[
\begin{gathered}
D^i = \left( \tfrac{\partial}{\partial x_1}\right)^{i_1} \cdots  \left( \tfrac{\partial}{\partial x_n}\right)^{i_n};
\quad p(D) = \sum_{i\in\mathcal I(n,d)} c(i)a(p;i)D^i.
\end{gathered}
\]
Then it is not hard to show that for $p,q \in F_{n,d}(\mathbb C)$,
\[
p(D)\bar q = \bar q(D)p = d![p,\bar q]; \qquad ||p||^2 = \tfrac 1{d!} p(D)\bar p.
\]

In 1990 \cite[Thm 1.2]{BBEM}, 
Beauzamy, Bombieri, Enflo and Montgomery proved a remarkable inequality involving
a norm on products of polynomials (see also \cite{Rez-5}, \cite{Zeil}). The inequality
can be put into either of two equivalent forms:

\begin{prop}[\cite{BBEM, Rez-5}] \label{T:BBEM}
Suppose $f \in F_{n,e_1}(\mathbb C)$ and $g \in F_{n,e_2}(\mathbb C)$, then
\[
\begin{gathered}
||fg|| \ge \sqrt{\frac{(e_1+e_2)!}{e_1!e_2!}}\cdot  ||f||\cdot ||g||; \\
(fg)(D)(\overline{fg}) \ge (f(D)\bar f)\cdot (g(D)\bar g).
\end{gathered}
\]
\end{prop}

Now we turn to the perturbation result. 
For $s > 0$, let
\[
E_s = E_s(n,d) = \Biggl\{\sum_{i \in \mathcal I(n,d)} c(i)\la_i x^i ; \quad  |\la_i| \le s \Biggr\}.
\]
\begin{thm}\label{irred}
Suppose $p \in \mathbb C[x_1,\dots,x_n]$ is an irreducible form of degree $d$. Then for sufficiently
small $s > 0$, every form $p + q, q \in E_s$ is irreducible.
\end{thm}

\begin{proof}[Proof of Theorem \ref{irred}]
First suppose $p(1,0,\dots,0) \neq 0$, and by taking a multiple, assume without loss of generality that $p(1,0,\dots,0) = 1$. It follows that $||p|| \ge 1$.  

Suppose there are reducible forms arbitrarily close to $p$. This means that there is a decreasing sequence $s_m\to 0$ for which there exist
reducible $p_m = p + q_m$ with $q_m \in E_{s_m}$. Without loss of generality, we may assume that $s_m < \frac12 \cdot  n^{-d/2}$ for all $m$.

Write $p_m = f_mg_m$, where
$deg(f_m) = j_m\ge 1$ and $deg(g_m) = d - j_m\ge 1$
Since $j_m \in \{1,\dots,d-1\}$,
it has only finitely many possible values, and so at least one degree $j$ occurs infinitely often. 
Now take a subsequence so that $s_{m(\ell)} \to 0$ and for which 
$deg(f_{m(\ell)}) = j \ge 1$ for all $m(\ell)$, so  $deg(g_{m(\ell)}) = d -j \ge 1$
 and (if necessary) rename the subsequence. Thus, we have a  sequence $s_m \to 0$, $q_m \in E_{s_m}$, 
 and forms $f_m, g_m $, of degree $j, d-j$ so that
 \begin{equation}\label{E:fac}
p_m = p + q_m = f_m*g_m.
 \end{equation}
 Since $s_m \le \frac 12 \cdot  n^{-d/2}$, by \eqref{E:bound}, we have 
\begin{equation}
||p + q_m|| \le ||p|| + ||q_m|| = ||p|| + \tfrac 12 \le 2||p||.
\end{equation}
Furthermore, since $q_m \in E_{s_m}$, $|a(q;(de_1))| \le s_m$, hence $|q_m(1,0,\dots,0)| \le s_m \le \frac 12$, thus
\begin{equation}
1 + s_m \ge |p_m(1,0,\dots,0)|  \ge 1 - s_m \implies \tfrac 32 \ge  |p_m(1,0,\dots,0)|  \ge \tfrac 12.
\end{equation}
We may scale $f_m, g_m$ in \eqref{E:fac} to assume that $f_m(1,0,\dots,0) = g_m(1,0,\dots,0) = \sqrt{p_m(1,0\dots,0)}.$
That is
\begin{equation}
a(f_m;je_1) = a(g_m;(d-j)e_1)=\sqrt{p_m(1,0\dots,0)}.
\end{equation}
It follows then that for all $m$,
\begin{equation}
||f_m|| \ge \sqrt{1/2}, \qquad ||g_m|| \ge \sqrt{1/2}.
 \end{equation}

On the other hand,   $||p_m|| \le 2||p||$, and so,
by Proposition \ref{T:BBEM}, with $f= f_m$ and $g = g_m$ and $fg =p_m$, we have
\begin{equation}
\begin{gathered}
2||p|| \ge ||p_m|| \ge \sqrt{\frac{d!}{j!(d-j)!}}\cdot  ||f_m||\cdot || g_m||  \\
\implies ||f_m|| \le \sqrt{\frac{j!(d-j)}{d!}}\cdot \frac{2||p||}{||g_m||}  \le 
2\sqrt 2\cdot ||p||\cdot \sqrt{\frac{j!(d-j)!}{d!}}.
\end{gathered}
 \end{equation}
This implies that $||f_m||$ is uniformly bounded (the actual value is unimportant) as $m \to \infty$,
say $||f_m|| \le T$. Identical reasoning shows that $||g_m|| \le T$ as well.

But now, by \eqref{E:bound}, $|a(f_m,i)|, |a(g_m;i)| \le \sqrt T$. Consider the $N(n,j)+N(n,d-j)$-tuples of the
coefficients of $f_m$  followed by the coefficients of $g_m$. It follows from Lemma \ref{L:BW} that 
there is a convergent subsequence of $(f_m,g_m)$'s. That is, there exist forms $f$ of degree $j$ and
$g$ of degree $d-j$ so that
\begin{equation}
\lim _{k\to\infty} f_{m_k}(x_1,\dots x_n) = f(x_1,\dots,x_n), \quad \lim _{k\to\infty} g_{m_k}(x_1,\dots x_n) = g(x_1,\dots,x_n).
 \end{equation}

Finally,  since
\begin{equation}
p_m(x_1,\dots,x_n)= f_m(x_1,\dots,x_n) g_m(x_1,\dots,x_n)
 \end{equation}
we see that the following limit exists:
\begin{equation}
\lim _{k\to\infty} p_{m_k}(x_1,\dots,x_n) = f(x_1,\dots,x_n) g(x_1,\dots,x_n), 
 \end{equation}
where $\deg f= j \ge 1, \deg g = d-j \ge 1$. 
We have
\begin{equation}
p_{m_k} = p + q_{m_k} \implies a(p_{m_k};i) = a(p;i) + \la_i, \quad |\la_i| < s_{m_k}.
 \end{equation}
It follows that for $i \in \mathcal I(n,d)$, $\lim_m a(p_{m_k};i) = a(p;i)$; Thus
\begin{equation}
p(x_1,\dots,x_n) =  \lim_k p_{m_k}(x_1,\dots,x_n)  = f(x_1,\dots,x_n)g(x_1,\dots,x_n).
 \end{equation}
is reducible over $\mathbb C$. This contradiction proves the
theorem in case $p(1,\dots,0) \neq 0$.

Finally, suppose $p$ is a non-zero irreducible form but $p(1,0,\dots,0) = 0$. By making an
invertible change of  variables in $p$, let 
\[
\tilde p(x) = p\left(\sum_{j=1}^n a_{1j}x_j, \dots, \sum_{j=1}^n a_{nj}x_j\right),
\]
so that $\tilde p(1,0,\dots,0) = 1$
and apply the argument above to $\tilde p$, so that any small perturbation of $\tilde p$ is irreducible.  
A continuity argument shows that if $q \in E_s$, then for some
constant $M$,
\[
q\left(\sum_{j=1}^n a_{1j}x_j, \dots, \sum_{j=1}^n a_{nj}x_j\right) \in E_{M\cdot s}.
\]
Thus, we apply the result to perturbations of $\tilde p(x)$; if we perturb $p$ by $q \in E_{s/M}$, then the
sum is still irreducible.
\end{proof}

We shall also need the fact that $x_1^d - x_2^d + x_3^d$ is irreducible. 
This result and its proof are both taken from \cite[Thm. 4.1]{P-R}.
\begin{prop}\label{T:PR}
Suppose $f(x,y)$ is a square-free binary form of degree $d$ in $\mathbb C[x,y]$. Then
$F(x,y,z) = f(x,y) + z^d$ is irreducible over $\mathbb C$. In particular, $x^d-y^d+z^d$ is
irreducible.
\end{prop}
\begin{proof}
Suppose otherwise that $F = GH$, where $\deg G = j$, $\deg H = d-j$ where $1 \le j \le d-1$.
Since $F(0,0,1) = 1$, both $G$ and $H$ must have terms involving $z$. Write $G$ and $H$
as polynomials in increasing powers of $z$:
\[
\begin{gathered}
G(x,y,z) = G_0(x,y) + G_a(x,y)z^a + \dots \\
H(x,y,z) = H_0(x,y) + H_b(x,y)z^b + \dots 
\end{gathered}
\]
Here, $G_a$ and $H_b$ are non-zero binary forms of degree $j-a$ and $n-j-b$ respectively.
We see that $F(x,y,0) = f(x,y) = G_0(x,y)H_0(x,y)$, and since $f$ is square-free, 
$gcd(G_0,H_0) = 1$. If $a < b$ then $F = GH$ will contain the un-cancelled term
$H_0(x,y)G_a(x,y)z^a$, and since $a \le j \le d-1$, this is impossible. A similar contradiction
occurs if $b < a$, hence $a=b$, and the coefficient of $z^a$ in $GH$ is $G_0H_a + H_0G_a$.
Since $a < d$, we must have $G_0H_a + H_0G_a = 0$, and since $gcd(G_0,H_0) = 1$, this
implies that $G_0 \ | \ G_a$. But $deg(G_0) = j$ and $
deg(G_a) = j-a$, so $G_a = 0$, a contradiction. In particular, $x^d - y^d$ is square-free, 
completing the proof.
\end{proof}


By using 
the topological fact Theorem \ref{irred} 
together with the indefiniteness and irreducibility of $x_1^d - x_2^d + x_3^d$
from Proposition \ref{T:PR}, we are able to give a quick proof of Theorem \ref{T:thmB} for $n \ge 3$.

\begin{proof}[Proof of Theorem \ref{T:thmB}]
For $n=2$, this has already been done. Now suppose $n \ge 3$. In view of Corollary \ref{C:square}, it
suffices to find indefinite irreducible forms in $F_{n,d}$ with exactly $k$ terms. 
As we have seen, the real form $p(x) = x_1^d - x_2^d + x_3^d$ is irreducible; since
$p(e_1) = 1$ and $p(e_2) = -1$, it is indefinite.

Take a sufficiently small real perturbation of $p$  involving exactly $j$ ($0 \le j \le N(n,d) - 3$) monomials
other than $x_1^d, x_2^d, x_3^d$. This produces an indefinite irreducible form in 
$F_{n,d}^k$ with exactly $k$ monomials, for every value $k = 3+j$, where
 $0 \le j \le N(n,d) - 3$. This covers all needed values of $k$, except for $k=2$.
 
For $k=2$, suppose $p(x_1,\dots,x_n) = x_1x_2^{d-1}+x_3^d$. Then $p(1,1,\dots,0) = 1$ and
$p(-1,1,0,\dots,0) = -1$, so $p$ is indefinite. Suppose $p = qr$, where $q,r$ are not constants.
The degree of $x_1$ in $p$ equals 1, which  is the sum of the degrees of $x_1$ in $q$ and $r$,
hence after permuting $q$ and $r$ if necessary, we may assume that $x_1$ does not appear in $q$. Thus
\[
p(x_1,\dots,x_n) = x_1x_2^{d-1}+x_3^d = q(x_2,\dots, x_n)\left(r_1(x_2, \dots, x_n) x_1+ r_2(x_2, \dots, x_n)\right).
\]
It follows that $q$ divides both $x_2^{d-1}$ and $x_3^d$, so $q$ is constant, a contradiction. Hence $p$ is irreducible and $p \in F_{n,d}^2 \setminus  F_{n,d}^{1}$.

In other words, we have constructed  indefinite irreducible 
$p \in F_{n,d}^k \setminus  F_{n,d}^{k-1}$. Theorem \ref{T:thmB} then follows from Corollary \ref{C:square}.
\end{proof}

We conclude with some final remarks suggesting future work.

First, since each $\Sigma_{n,d}^k$ is a closed convex cone, the strict inclusion
implies that $\Sigma_{n,2d}^k \setminus \Sigma_{n,2d}^{k-1}$ has an interior. 
(Robinson \cite[\S3]{Rob} showed in 1969 
that any sufficiently small perturbation of $\sum x_j^{2d}$ is in $\Sigma_{n,2d}^2$.)
Thus, beyond the examples given here, there
should exist forms which are not perfect squares there. 
More directly, pick $p \in \Sigma_{n,2d}^k \setminus \Sigma_{n,2d}^{k-1}$. If $p$ is not a perfect square, then we are done!
Otherwise, suppose $p = g^2$ for some form $g$.
We may assume that $p$ is not a multiple of $x_1^{2d}$ (otherwise, permute variables).
Since $\Sigma_{n,2d}^{k-1}$ is closed, there exists $\lambda > 0$ so that 
$g^2 + \lambda x_1^{2d} \in \Sigma_{n,2d}^k \setminus \Sigma_{n,2d}^{k-1}$ as well. 
If any such form is a square, then
\[
h^2 = g^2 + \lambda x_1^{2d} \implies  \lambda x_1^{2d} = h^2 - g^2 = (h+g)(h-g) \implies
h \pm g = c_{\pm}x_1^{d},
\]
which implies that $p = g^2$ is  a multiple of $x_1^{2d}$, a contradiction.  It would be useful to find explicit families of forms in
$\Sigma_{n,2d}^k \setminus \Sigma_{n,2d}^{k-1}$.

We note that $\Sigma^1_{n,2d}$ corresponds to linear programming,  $\Sigma^2_{n,2d}$ corresponds to second order 
cone programming (see \cite{Ahmadi-Majumdar}) and $\Sigma^{N_{n,d}}_{n,2d} = \Sigma_{n,2d}$ corresponds to semidefinite programming. 
It seems natural to investigate what $\Sigma^k_{n,2d}$ corresponds to for $3 \le k < N(n,d)$.

Finally, it seems important to examine the asymptotic behavior of the cones $\Sigma^k_{n,2d}$ as $n \to \infty$.

\end{document}